\documentclass{article}
\usepackage{amsfonts,amsmath,amsthm}

\usepackage{hyperref}
\usepackage[capitalise, noabbrev]{cleveref}

\newcommand{\ZZ} {\mathbb{Z}}
\newcommand{\FF} {\mathbb{F}}

\newtheorem{lemma}{Lemma}
\newtheorem{conjecture}{Conjecture}
\newtheorem{theorem}[lemma]{Theorem}
\newtheorem{corollary}[lemma]{Corollary}
\newtheorem{proposition}[lemma]{Proposition}
\newtheorem{definition}[lemma]{Definition}
\newcommand\latticebottom{{\bf 0}}
\newcommand\lV{{\cal L}(V)}
\newcommand\card[1]{\left|\, #1 \, \right|}
\newcommand\bigcard[1]{\Bigl| \, #1 \, \Bigr| }
\newcommand\st{\,|\,}
\newcommand\colonst{\,:\,}
\DeclareMathOperator{\sgn}{sgn}

\newcommand\linspan[1]{\left\langle #1 \right\rangle}
\newcommand\upto{\mathbin{\ldotp\ldotp}}

\hypersetup{
  pdfinfo={
    Title={A theory of q-transversals},
    Subject={Combinatorics},
    Author={Mark Saaltink}
  }
}

\title{A theory of $q$-transversals}
\author{Mark Saaltink}
\date{March 14, 2025}

\begin{document}
\maketitle
\begin{abstract}
   Given an indexed family ${\cal A} = (A_1, A_2, \dotsc, A_n)$ of subsets of some given set $S$, a \emph{transversal} is a set of distinct elements $x_1, x_2, \dotsc, x_n$ with each $x_i \in A_i$.  Transversals have been studied since 1935 and have many attractive properties, with a deep connection to matroids.
A $q$-analog is formed by replacing the notion of a set by the notion of a vector space, with a corresponding replacement of other concepts.  In this paper we define a $q$-analog of the theory of transversals, and show that many of the main properties of ordinary transversals are shared by this analog.
\end{abstract}

\section{Introduction}

A $q$-analog is formed by replacing the notion of a set by the notion of a vector space, with a corresponding replacement of other concepts: cardinality becomes dimension, elements are replaced by 1-dimensional subspaces, the empty set is replaced by the 0-dimensional subspace, intersection remains the same, and union is replaced by sum.  (It is not always clear how to replace set difference or symmetric difference.) The application of this idea to matroid theory gives the still-young theory of $q$-matroids~\cite{q-matroids}.


In this paper I define a $q$-analog of a \emph{transversal}~\cite{mirsky1971}, and show that there is still a connection with matroids, as well as analogs of many of the main theorems of transversal theory.

\section{Transversal theory}

We begin with a review of standard definitions and main results of transversal theory.  More comprehensive treatments of the subject can be found in~\cite{bonin2010,mirsky1971,welsh1971,welsh-matroid-theory}.
%
We are given an indexed family ${\cal A} = (A_1, A_2, \dotsc, A_n)$ of subsets of some given set $S$.  A \emph{transversal} is a set of distinct elements $x_1, x_2, \dotsc, x_n$ with each $x_i \in A_i$.  A \emph{partial transversal} is a transversal of some subfamily (that is, containing some subset of the $A_i$).

Hall was the first to show that the obviously necessary condition is sufficient for the existence of a transversal.
It is conventional to define $A[J] = \bigcup_{j\in J} A_j$, and we will write $\{1 \upto k\}$ for the set $\{1, 2, \dotsc, k\}$.
\begin{theorem}[Hall, 1935 \cite{hall1935}] \label{th:hall}
  $\cal A$ has a transversal iff for every $J \subseteq \{1 \upto n\}$ we have
    \[ \card { A[J] } \geq \card J. \]
\end{theorem}
This has a $q$-analog in \cref{th:q-Hall} below.  Hall's theorem was later generalized by Rado.

\begin{theorem}[Rado, 1967 \cite{rado1967}] \label{th:rado}
  Suppose $M$ is a matroid on the set $S$, with rank function $\rho$.
  $\cal A$ has a transversal that is independent in $M$
  iff for every $J \subseteq \{1 \upto n\}$ we have
    \[ \rho (A[J]) \geq \card J. \]
\end{theorem}
When $M$ is the free matroid on $S$, then for all $X \subseteq S$ we have $\rho(X)=\card X$, so we get Hall's theorem as a special case.
Rado's theorem does not yet have a known $q$-analog.


In the mid 1960s, the structure of the set of transversals was recognized.  The proofs of many properties of transversals become simple once this structure is shown.  This has a $q$-analog in \cref{th:transversals-are-matroids}.
\begin{theorem}[Edmonds and Fulkerson, 1965 \cite{edmonds1965}] \label{th:edmonds-fulkerson}
  The set of partial transversals of a given family are the independent sets of a matroid.
\end{theorem}


In light of \cref{th:edmonds-fulkerson} we say that a matroid $M$ is a \emph{transversal matroid} if there is some family $\cal A$ so that the $M$'s independent sets are exactly the partial transversals of $\cal A$, and $\cal A$ is then called a  \emph{presentation} of $M$.
\begin{theorem}[Various authors] \label{th:presentation-properties} \label{th:transversal-flats}
  \begin{enumerate}
  \item If matroid $M$ has a presentation, then it has a presentation with exactly $\rho(M)$ sets.
  \item If transversal matroid $M$ has no coloops, then every presentation has exactly $\rho(M)$ sets.
  \item If matroid $M$ has a presentation $\cal A$, some subfamily of $\cal A$  with exactly $\rho(M)$ sets is also a presentation of $M$.
  \item $(A_1, \dots, A_r)$ is a presentation of $M$ and $r = \rho(M)$, then for all $i$ the set $E(M) \setminus A_i$ is a flat. \label{item:transversal-flats}
  \item If $M$ is a transversal matroid, it has a unique \emph{maximal} presentation~\cite{mason1970}, where a presentation of $M$ is maximal if adding any element to any of its sets gives a family that is \emph{not} a presentation of $M$.
  \end{enumerate}
\end{theorem}
$q$-analogs of these properties appear in \cref{sec:q-presentations}.

Transversal matroids belong to the class of representable matroids.
\begin{theorem}[Piff and Welsh\cite{piffwelsh1970}] \label{th:transversals-rep}
  Transversal matroids are representable over all sufficiently large fields.  Thus they are representable over fields of every characteristic.
\end{theorem}
This has a $q$-analog in a special case in \cref{th:q-rep-aligned}, and we conjecture in
\cref{sec:rep} that a more general analog holds.

Transversal matroids are a special type of matroid union.  This has a $q$-analog in \cref{thm:q-matroids-union-rank-1}.
\begin{theorem}[Mirsky \cite{mirsky1971}] \label{thm:matroids-union-rank-1}
A matroid is a transversal matroid iff it is the union of matroids of rank 1.
\end{theorem}

Proofs of all these theorems and more can be found in~\cite{mirsky1971,welsh1971,welsh-matroid-theory}.  In particular, Mirsky~\cite{mirsky1971} has three different proofs of Hall's theorem and gives many of its variants.  A fairly modern survey of transversal theory can be found in an unpublished article by Bonin~\cite{bonin2010}.


As a consequence of Hall's theorem, we have an alternative test for being a transversal.  This has a $q$-analog in \cref{th:is-q-transversal}.
\begin{corollary} \label{th:is-transversal}
  A set $T$ with $\card T = n$ is a transversal of family $\cal A$ iff for every $J \subseteq \{1 \upto n\}$ we have
  \[ \left| T \cap  A[J] \right| \geq \card J. \]
\end{corollary}
\begin{proof} The condition is clearly necessary.  So show it sufficient,
  form a family $\cal B$ with $B_i = T \cap A_i$. Then Hall's theorem shows that $\cal B$ has a transversal.  All the elements of that must be elements of $T$, and as both the transversal and $T$ have $n$ elements, they are the same.
\end{proof}
This test requires checking $2^n$ sets $J$, and so might be better than checking $n!$ different injections from $\{1 \upto n\}$ to $T$.  
This theorem has a $q$-analog in \cref{th:is-q-transversal}.


The class of transversal matroids is not minor-closed; in particular a contraction of a transversal matroid need not be transversal.  The larger class of \emph{gammoids} is the smallest minor-closed class containing all transversal matroids; details can be found in \cite{bonin2010}, \cite{oxley-matroid-theory},  or
\cite{welsh-matroid-theory}.

\subsection{A cryptomorphism} \label{sec:avoidance}

It turns out that a different formulation of transversal theory, while equivalent to the standard one, has a closer similarity to the $q$-analog.  Here we will give that formulation, which considers the complements of the sets $A_i$, and restate the main theorems in this setting.

\begin{definition}
 Given an indexed family ${\cal X} = (X_1, X_2, \dotsc, X_n)$ of subsets of some given set $S$, an \emph{avoiding transversal of $\cal X$} is a set of elements $x_1, x_2, \dotsc, x_n$ of $S$, all distinct, with each $x_i \notin X_i$.  A \emph{partial avoiding transversal of $\cal X$} is an avoiding transversal of some subfamily (that is, containing some subset of the $X_i$).
\end{definition}

Clearly, when $A_i = S \setminus X_i$, an avoiding transversal of $\cal X$ is just a transversal of $\cal A$.
From \cref{th:edmonds-fulkerson} it is clear that the set of avoiding transversals of a family form the independent sets of a matroid.  Similar to the definition of $A[J]$, it is useful to
define \[ {\cal X}(J) =  \bigcap_{j\in J} X_j \] for any $J \subseteq \{1 \upto n\}$, with the convention ${\cal X}(\emptyset) = S$.


Hall's theorem has a simple reformulation
\begin{corollary}\label{th:avoid-hall}
 $\cal X$ has an avoiding transversal iff for every $J \subseteq \{1 \upto n\}$ we have
   \[ |{\cal X}(J)| + \card J \leq \card S. \]

\end{corollary}

For the statement of Rado's theorem in this setting, it is convenient to recall first the \emph{nullity} function for a matroid $M$ with rank function $\rho$,
\[ \nu(A) = \card A - \rho(A), \]
and recall an identity for rank in a dual matroid,
\[ \rho^*(S \setminus X) = \nu(S) - \nu(X) \text{~and dually~} \rho(S \setminus X) = \nu^*(S) - \nu^*(X). \]
Now if each $X_i = S \setminus A_i$, we have
\[ \rho\left( A[J] \right)
= \rho\left( S \setminus {\cal X}(J) \right)
= \nu^* (S) - \nu^*\left( {\cal X}(J) \right). \]
Co-nullity is not as mysterious as it may first appear.  If $M$ is a matroid on $S$ with rank function $\rho$ and $X \subseteq S$, then we have
\begin{align*}
  \rho(X) &= \max_{B \in \cal B} \card {X \cap B}, \text{~where $\cal B$ is the set of bases of $M$, and so} \\
  \nu(X) &= \min_{B \in \cal B} \card {X \setminus B}.
\end{align*}
Since the bases of $M^*$ are the sets $S \setminus B$, we have
\begin{equation} \label{eq:nu*}
  \nu^*(X) = \min_{B \in \cal B} \card {X \setminus (S \setminus B)}
  = \min_{B \in \cal B} \card {X \cap B}.
\end{equation}
So $\rho$ and $\nu^*$ have a pleasing similarity, replacing max by min.

We can recast Rado's theorem using co-nullity:
\begin{corollary} \label{th:avoid-rado}
  Suppose $M$ is a matroid on the set $S$, with co-nullity function $\nu^*$.
  Then $\cal X$ has an avoiding transversal that is independent in $M$
  iff for every $J \subseteq \{1 \upto n\}$ we have
  \[ \nu^*\left( {\cal X}(J) \right)  + \card J \leq \nu^* (S). \]
\end{corollary}
When $M$ is the free matroid, the only basis is $S$ and co-nullity coincides with cardinality, so this version of Rado's theorem implies the avoidance version of Hall's theorem.

\cref{th:is-transversal} can also be expressed in the avoidance form, although it is fact easy to express a strengthened version.  This theorem has a $q$-analog in \cref{th:is-q-transversal}.
\begin{theorem} \label{th:is-avoiding-transversal}
  A set $T$ is a partial avoiding transversal of family ${\cal X} = (X_1, \dotsc, X_n)$ iff for every $J \subseteq \{1 \upto n\}$ we have
    \[ | T \cap {\cal X}(J) | + \card J \leq n. \]
\end{theorem}
\begin{proof}
  If $\card T > n$ then $T$ cannot be a transversal, and with $J = \emptyset$ the condition is falsified.  So the theorem holds in that case.  Now assume $\card T \leq n$.
  Let $U$ be a set of size $n - \card T$ that is disjoint from $T$ and all the $X_i$, put $T' = T \cup U$, and let $A_i = T' \setminus X_i$ for $i \in \{1 \upto n\}$.  Now $T$ is a partial avoiding transversal of $\cal X$ iff $T'$ is a transversal of $(A_1, A_2, \dotsc)$.
Then for all $J \subseteq  \{1 \upto n\}$ we have
  \[ \card{A[J]} = \Bigl| \bigcup_{j\in J} (T' \setminus X_j) \Bigr| =
  \Bigl| T' \setminus \bigcap_{j\in J} X_j \Bigr| = n - |T \cap {\cal X}(J)|, \]
  so that Hall's theorem completes the proof.
\end{proof}

The properties of \cref{th:presentation-properties} of course hold, with a slight recasting to refer to an ``avoiding presentation'', and to refer to a minimal avoiding presentation, rather than maximal presentation.  Most of these properties have $q$-analogs in \cref{sec:q-presentations}. 

\section{$q$-matroids}

Just as transversals are connected with matroids, $q$-transversals will be connected to $q$-matroids.  In this section we review the basic definitions of these matroids, including just enough of the theory for our needs.
While there are many equivalent definitions of $q$-matroids~\cite{q-alternatives,q-matroids}, the simplest uses the rank function.  When $V$ is a vector space, we let $\lV$ be the lattice of subspaces of $V$.
A \emph{$q$-matroid} is determined by a vector space $V$ and a function $r: \lV  \rightarrow \ZZ$,
satisfying, for all $A, B \in \lV$,
\begin{itemize}
  \item $0 \leq r(A) \leq \dim A$ (rank is non-negative and bounded by dimension),
  \item if $A \leq B$ then $r(A) \leq r(B)$ (rank is non-decreasing), and
  \item $r(A \vee B) + r(A \wedge B) \leq r(A) + r(B)$ (rank is submodular).
\end{itemize}
This definition is a direct $q$-analog of the definition of an ordinary matroid in terms of its rank function.
Here we use lattice notations: $A \leq B$ holds if $A$ is a subspace of $B$, $A \vee B$ is the smallest subspace containing both
$A$ and $B$, or equivalently the space spanned by their union, and $A \wedge B = A \cap B$ is the largest space contained in both $A$ and $B$.
We also write $\latticebottom$ for the bottom element of $\lV$, that is, the subspace of dimension 0.
If $X$ a set of vectors, we will write $\linspan X$ for their linear span.

The usual concepts of matroid theory apply equally well to $q$-matroids: a subspace $A$ is \emph{independent} if $r(A) = \dim A$, and \emph{dependent} otherwise; a subspace $A$ is a \emph{circuit} if it is dependent but all its proper subspaces are independent; the \emph{closure} of $A$ is the largest subspace $B$ with $A \leq B \leq V$ and $r(B) = r(A)$; and $A$ is \emph{spanning} if its closure is $V$.  Axiom systems for $q$-matroids in terms of independent sets, bases, or circuits can be found in \cite{q-alternatives}.

If $A$ and $B$ are subspaces, and $r(A) = r(B) = 0$, then by submodularity we also have $r(A \vee B) = 0$.   So there is a largest subspace $L \leq V$ with $r(L) = 0$, called the \emph{loop space} of the $q$-matroid.  The loop space is the closure of the 0-dimensional subspace at the bottom of the lattice.

Matroids of rank 1 figure prominently in transversal theory.  A $q$-matroid of rank 1 is completely characterized by its
loop space: if $M$ is a $q$-matroid of rank 1 with loop space $L$ then for any $A \leq V$ we have
\[ r(A) = \begin{cases}
    0 & \text{if $A \leq L$}, \\
    1 & \text{otherwise.}
\end{cases} \]

As for ordinary matroids, we can define the \emph{nullity} function for a $q$-matroid,
\[ n(X) = \dim X - r(X). \]
Nullity is non-negative, not greater than dimension, monotonic, and \emph{super}modular.

\begin{proposition} \label{th:fundamental-circuit}
  Let $M$ be a matroid on vector space $V$ with nullity function $n$.  If $S \leq V$ has nullity 1,
  then there is a unique circuit $C$ satisfying $C \leq S$.  Moreover, for any $T \leq S$ we have
  \[ n(T) = \begin{cases} 1 & \text{if $C \leq T$} \\ 0 & \text{otherwise.} \end{cases} \]
\end{proposition}
\begin{proof}
Suppose $A, B \leq S$ both have nullity 1.  Then by monotonicity $A \vee B \leq S$ also has nullity 1, and so by supermodularity, $A \wedge B$ has nullity at least 1.  The nullity cannot be larger, so $A \wedge B$ has nullity 1.  Thus the meet of all subspaces of $S$ of nullity 1, (of which there is at least one, namely $S$ itself) is the smallest dependent subspace of $S$. Call that $C$.  Then all proper subspaces of $C$ are independent, so that $C$ is a circuit, and all subspaces of nullity 1 contain $C$.
\end{proof}
This gives an analog of fundamental circuit:
if $B \leq V$ is a basis and $a$ has dimension 1 but is not in $B$, $B \vee a$ has nullity 1 and so contains a unique circuit $C$.  This can be declared as the fundamental circuit of $a$ in $B$.  Unlike the case of ordinary matroids, we may not have $a \leq C$.

\subsection{Induction and union}

A function $f: \lV \rightarrow \ZZ$ is \emph{submodular} if for all $A, B \in \lV$ we have
  \begin{itemize}
  \item $f(\latticebottom) = 0$ (where $\latticebottom$ is the empty subspace),
  \item if $A \leq B$ then $f(A) \leq f(B)$, and
  \item $f(A \wedge B) + f(A \vee B) \leq f(A) + f(B)$.
  \end{itemize}
  That is, $f$ satisfies all the properties of a rank function except it may not be bounded by a subspace's dimension. (Caution: some authors do not require $0 \leq f(A)$.)

Given such a function $f$, there is an \emph{induced} $q$-matroid on $\lV$:  subspace $A$ is independent iff for all $B \leq A$ we have $f(B) \geq \dim B$.
The rank of an element $A$ is the maximum of $\dim B$ for any independent $B \leq A$.
An explicit formula for the rank function 
is
\[ r(A) = \min_{B \leq A} f(B) + \dim A - \dim B. \]

When $q$-matroids $M$ and $M'$ are both defined on the same lattice $\lV$,
we write $M \vee M'$ for the union of $q$-matroids as described in~\cite{q-sum}. That is,
we define function $f: \lV \rightarrow \ZZ$ by the equation $f (A) = r(A) + r'(A)$ for all $A \leq V$;  then $f$ is submodular, and $M \vee M'$ is the $q$-matroid it induces.
It is easily shown that in determining $(M_1 \vee M_2) \vee M_3$, we can just use one step and induce from $f(A) = r_1(A) + r_2(A) + r_3(A)$, and similarly for larger numbers of summands.



\begin{proposition} \label{th:union-same-rank}
Suppose $M$ and $N$ are $q$-matroids on the same space $V$. If $M \vee N$ has the same rank as $M$, then $M \vee N = M$.
\end{proposition}
\begin{proof}
  Let $r_M$ be the rank function of $M$ and $r_N$ be the rank function of $N$, 
  so that $M \vee N$ is induced by function $f = r_M + r_N$.
  Obviously, any $S \leq V$ that is independent in $M$ is also independent in $M \vee N$, as for all $X \leq S$ we have $f(X) \geq r_M(X) = \dim X$.  If we can show that all circuits on $M$ are dependent in the union, we will be done.  So let $C \leq V$ be a circuit of $M$.  Then $C$ is the irredundant join of some 1-dimensional subspaces; let $a$ be one of them and let $I$ be the join of the rest.  Then $C = I \vee a$ with
  $\dim I = \dim C - 1$.  Then $I$ is independent  in $M$ and so can be extended to a basis $B$ of $M$.
  We cannot have $B \vee a = B$, as then $C \leq B$ and $B$ would not be independent.  So $B \vee a$ has nullity 1 in $M$, and $C$ is the corresponding fundamental circuit in $M$ (\cref{th:fundamental-circuit}).  By assumption $B\vee a$ is not independent in $M \vee N$ (or the union's rank would be greater than $M$'s; $\dim (B \vee a)$ is one more than the rank of $M$), so there exists some $X \leq B \vee a$ with $f(X) < \dim X$.  As $f (X) \geq r_M(X)$, $X$ is not independent in $M$.  Since $X \leq B \vee a$, $X$ must have nullity 1 in $M$ and thus $C \leq X$.  Finally, $\dim X > f (X) = r_M(X) + r_N(X) = \dim X - 1 + r_N(X)$, so we have $r_N(X) = 0$.  But then as $C \leq X$ we have $r_N(C) = 0$.  So $f (C) = r_M(C) + r_N(C) = \dim C - 1$, and $C$ is also dependent in the union.
\end{proof}

\subsection{Representability}

A large class of $q$-matroids can be defined algebraically; these are the \emph{representable} $q$-matroids.
Given some $n\times m$ matrix $G$ over a field $K$ extending $\FF_q$,
we can define the $q$-matroid represented by $G$ as follows: if a subspace of $\FF_q^m$ is the row space of matrix $X$, then its rank in the matroid is the rank of matrix $GX^T$.
It can be seen that this does not depend on the matrix $X$ we use;
if $X$ and $Y$ are full-rank matrices with equal row spans, then there is some invertible matrix $A$ so that $Y = AX$.  Then $GY^T = GX^TA^T$, and as $A^T$ is invertible, the ranks are the same.

Note that the entries of $G$ may lie in a larger field than the vector space is defined over.
Usually $K$ is equal to $\FF_{q^k}$ for some $k$.
Unlike a normal matroid that might be representable over fields of many different characteristics, a $q$-matroid representation has its characteristic fixed by the vector space the $q$-matroid is defined on.

\section{$q$-transversals}

We are finally ready to define $q$-transversals and to describe their properties.
If we want a theory of $q$-transversals with anything like \cref{th:edmonds-fulkerson}, we are immediately drawn to the avoidance version of the theory.  Consider the case of a transversal of a single set $A$.  If the set of transversals is to be a $q$-matroid,  that matroid will have rank 1 and is characterized by some loop space $L \leq V$.
Thus any 1-dimensional subspace of $L$ is not a transversal, and any 1-dimensional space not in $L$ is a transversal.  Thus $A$ must be $V \setminus L$, which is not a subspace (unless $L$ is empty or all of $V$, two trivial cases).  So we will adapt the avoidance version of \cref{sec:avoidance}.

In this section we will be talking about the vector space $V$ and its subspaces as well as $q$-matroids, so there is room for some confusion when we say ``independent'' or ``basis''.  Therefore we will qualify these terms and speak of a ``vector basis'' or ``independent vectors'' when referring to $V$ and its elements, or a ``matroid basis'' or ``independent subspace'' when referring to aspects of a $q$-matroid.

\begin{definition}
  Let a vector space $V$ be given, together with a family ${\cal X} = (X_1, \dotsc, X_n)$
  of subspaces of $V$.  A \emph{$q$-transversal} is a dimension $n$ subspace $T$ of $V$ where every vector basis $B$ of $T$ has an ordering $B = \{b_1, \dots, b_n\}$
  such that $b_i \notin X_i$ for  $i \in \{1 \upto n\}$.  That is, every vector basis of $T$ is an avoiding transversal of $\cal X$.
  A \emph{partial $q$-transversal} is a dimension $n$ subspace $T$ where every vector basis of $T$ is a partial avoiding transversal of $\cal X$.
\end{definition}

There are clearly some arbitrary decisions in this definition, and
there is room to explore other possible definitions of a $q$-transversal.  The definition here gives satisfactory analogs of the main theorems of \cref{sec:avoidance}.

The definition of partial $q$-transversal deserves some note.  It would be more analogous to the conventional theory to declare a partial $q$-transversal to be a $q$-transversal of some subsystem of $\cal X$.  This turns out to be equivalent (see \cref{th:partial-q-transversal-equivalence}), but is more difficult to work with.

We begin with an analog of ~\cref{th:edmonds-fulkerson}:
\begin{theorem} \label{th:transversals-are-matroids}
  Given a family ${\cal X} = (X_i \colonst i \in I)$ where each $X_i$ is a subspace of $V$, the set of partial $q$-transversals of $\cal X$ form the independent subspaces of a $q$-matroid.
\end{theorem}
\begin{proof}  Without loss of generality, we suppose $I = \{1 \upto k \}$.
  For $i \in I$, let $M_i$ be the rank 1 $q$-matroid with loop space $X_i$, and let $r_i$ be the rank function for $M_i$.  
  Then for all $Y \leq V$ we have
  \[ r_i(Y) = \begin{cases} 0 & \text{if $Y \leq X_i$} \\ 1 & \text{otherwise.} \end{cases} \]
Let $M$ be the $q$-matroid union of the $M_i$, and define \[ f(Y) = r_1(Y) + \dots + r_k(Y). \]
  Then $f$ is a submodular function on $\lV$, and it induces $M$'s rank function $r$.  In particular, a subspace $X$ is independent in $M$
  iff $f(Y) \geq \dim Y$ for every subspace $Y \leq X$.  We will show that $M$'s independent sets are exactly the partial $q$-transversals of $\cal X$.

  So, suppose first that $T$ is a partial $q$-transversal.  Let $Y$ be an arbitrary subspace of $T$, and let $\{b_1, \dotsc, b_m \}$ be a vector basis for $Y$.  Extend that to a vector basis $\{b_1, \dotsc, b_m, b_{m+1}, \dotsc \}$ of $T$.  Then this basis is a (partial) avoiding transversal of $\cal X$, and there is an injection $\pi$ so that $b_i \notin X_{\pi(i)}$.  Then $r_{\pi(i)}(Y) = 1$ for $i = 1, \dotsc, m$, and so $f(Y) \geq m = \dim Y$.  This holds for all subspaces $Y$ of $T$, so $T$ is independent in $M$.

  Conversely, suppose $T$ is independent in $M$, and let $\{b_1, \dots, b_m\}$ be a vector basis of $T$.  Let $C_i = \{ j \in \{1 \upto k\} \st b_i \notin X_j \}$,
  so that \[ r_j(\linspan {b_i}) = \begin{cases}
    1 & \text{if $j \in C_i$} \\
    0 & \text{otherwise.} \end{cases} \]
  For any  $J \subseteq \{1 \upto m\}$, consider the subspace spanned by $B = \{ b_j \st j \in J \}$.  As this is a subspace of $T$, which is independent, it must also be independent, so $f (\linspan B) \geq \card J$.
  Now $\linspan B = \bigvee_{j \in J} \linspan {b_j}$, and so
  \begin{align*}
  r_i(\linspan B)
  &= \begin{cases} 0 & \text{if $b_j \leq X_i$ for all $j \in J$} \\ 1 & \text{otherwise} \end{cases} \\
  &= \begin{cases} 1 & \text{if $i \in \bigcup_{j\in J} C_j$} \\ 0 & \text{otherwise} \end{cases}.
  \end{align*}
Then $f(\linspan B)$ counts the number of $i$ that are in that given union; that is
\[ f(\linspan B) = \bigcard {\bigcup_{j\in J} C_j}. \]

As noted above $f(\linspan B) \geq \card J$, because $\linspan B \leq T$ and so is independent.  Thus
\[ \bigcard { \bigcup_{j \in J} C_j} \geq \card J. \]
    Now by Hall's theorem (\cref{th:hall}), $(C_i \colonst i\in\{1 \upto m \})$ has a traversal.  Say this is
    $t(1), \dotsc, t(m)$.  Then $t$ is an injective mapping from $\{ 1 \upto m \}$ into $\{1 \upto k\}$ and we have
    $t(i) \in C_i$, and so $b_i \notin X_{t(i)}$.  That is, every basis of $T$ is a partial avoiding traversal of $\cal X$, and so $T$ is a partial $q$-transversal.
\end{proof}

The proof clearly shows an analog of \cref{thm:matroids-union-rank-1}
\begin{corollary} \label{thm:q-matroids-union-rank-1}
   A $q$-matroid is $q$-transversal iff it is the union of a collection of rank 1 $q$-matroids.
\end{corollary}
We say ``collection'' here rather than ``set'' as there may be multiple copies of a $q$-matroid needed.

We also have a result corresponding nicely to Hall's (\cref{th:hall}), in its avoidance version \cref{th:avoid-hall}.
\begin{theorem} \label{th:q-Hall}
  System ${\cal X} = (X_1, X_2, \dotsc, X_n)$ has a $q$-transversal iff for every nonempty $J \subseteq \{1 \upto n\}$ we have
    \[ \dim {\cal X}(J)  + \card J \leq \dim V. \]
\end{theorem}
\begin{proof}
  Following the line of argument in the proof of \cref{th:transversals-are-matroids}, define rank functions $r_i$ and function $f$.  Then the matroid of partial $q$-transversals of $\cal X$ is induced by $f$ and its rank function is therefore
  \[ r(V) = \min_{T \leq V} f(T) + \dim V - \dim T. \]
$r(V)$ can be at most $n$; if we take $T=V$ we see $r(V) \leq f(V) \leq n$.
So $\cal X$ has a $q$-transversal iff $f(V) \geq n$, equivalently iff
  for all $T \leq V$ we have
  \[ f(T) + \dim V - \dim T \geq n. \]
  For any $T$ we have $f (T) = n - \card{ \{ i \st T \leq X_i \} }$, as each $r_i$ adds 1 to $f(T)$ unless $T$ is a subspace of $X_i$.  Now let $J =  \{ i \st T \leq X_i \}$.  Then we have
  $T \leq {\cal X}(J)$ 
  so that $\dim T \leq \dim {\cal X}(J)$
and
$f(T) = n - \card J$.
So we have
\[ f(T) + \dim V - \dim T \geq  n - \card J + \dim V - \dim {\cal X}(J), \]
with equality when $T$ equals ${\cal X}(J)$,
and this is always larger than or equal to $n$ under the conditions
stated in the theorem.
\end{proof}

\cref{th:is-transversal} in its avoidance form \cref{th:is-avoiding-transversal} also has a $q$-analog.
\begin{theorem} \label{th:is-q-transversal}
  A subspace $T$
  is a partial $q$-transversal of family ${\cal X} = (X_1, \dotsc, X_n)$ iff for every $J \subseteq \{1 \upto n\}$ we have
\[ \dim \left( T \cap {\cal X}(J) \right) + \card J \leq n  \]
\end{theorem}
\begin{proof}
  Let $M$ be the $q$-matroid of partial $q$-transversals of $\cal X$, and let $r$ be its rank function.
  A calculation similar to the proof of \cref{th:q-Hall} shows $r(T) = \dim T$ is equivalent to the given condition.
\end{proof}

\newcommand\barn{\overline{n}}
We define \emph{bar nullity} as an analog of co-nullity that will not require us to take a dual.
If $M$ is a $q$-matroid and $\cal B$ is the set of bases, then for $X \leq V$ we define
\[ \barn(X) = \min_{B \in {\cal B}} \dim (B \cap X). \]
The analogy with co-nullity should be clear; compare this definition to \cref{eq:nu*}.
With this function we conjecture a $q$-analog of Rado's theorem (\cref{th:rado}) in its avoidance form \cref{th:avoid-rado}.
\begin{conjecture}
    Suppose $M$ is a $q$-matroid on vector space $V$, with bar nullity function $\barn$, and $\cal X$ is a system of subspaces of $V$.  Then $\cal X$ has a $q$-transversal that is independent in $M$ iff for all $J \subseteq I$ we have
  \[ \barn ( {\cal X}(J) ) + \card J \leq \barn (V). \]
\end{conjecture}

\subsection{Presentations} \label{sec:q-presentations}

If $M$ is the $q$-matroid of $q$-transversals of a system $\cal X$ we will say that $\cal X$ is a \emph{presentation} of $M$.
We can always find a presentation so that the bases of the associated $q$-matroid are full transversals, not partial, analogous to parts 1 and 3 of \cref{th:transversal-flats}.
\begin{theorem}
If $M$ is a $q$-transversal matroid of rank $r$, then it has a presentation with $r$ members.
\end{theorem}
\begin{proof}
  Let ${\cal X} = (X_1, \dotsc, X_k)$ be a presentation, and let $M_i$ be the rank-1 $q$-matroid with
  loop space $X_i$.  We saw earlier that $M = M_1 \vee \cdots \vee M_k$.
  Now if we form this union one matroid at a time, any summand that does not increase the rank can be discarded, by \cref{th:union-same-rank}.  This leaves a series of $r$ of the $M_i$, and their loop spaces give a presentation of $M$.
\end{proof}

We can use this idea to show that the two apparently different notions of partial $q$-transversal are the same.
\begin{corollary}\label{th:partial-q-transversal-equivalence}
$T$ is a partial $q$-transversal of ${\cal X} = (X_1, \dotsc, X_n)$ iff it is a $q$ transversal of some subsystem.
\end{corollary}

\begin{proof}
  The ``if'' part is trivial.  For the ``only if'' part,
  for $1 \leq i \leq n$ let $N_i$ be the matroid of partial $q$-transversals of the subsystem $(X_1, \dotsc, X_i)$, with rank function $\hat r_i$. So $T$ has rank $\hat r_n(T)$ in the matroid associated with $\cal X$, and as $T$ is a partial $q$-transversal we have $\hat r_n(T) = \dim T$.
As the order in which the $X_i$ are listed is irrelevant, we can assume they are listed in order so that $\hat r_i(T) = \min(i, \hat r_n T)$. We can do this by moving any $X_i$ that does not raise the rank of $T$ to the end.  Then $T$ is a basis of $N_j$ where $j=\dim T$, and $T$ is a transversal of $(X_1, \dotsc, X_j)$. 
\end{proof}

There is a simple $q$-analog of \cref{item:transversal-flats} of \cref{th:transversal-flats}.
\begin{theorem} \label{th:q-transversal-flats}
  If $M$ is the $q$-matroid of partial $q$-transversals of $(X_1, \dotsc, X_n)$ and $1 \leq i \leq n$, then $X_i$ is a flat of $M$.
\end{theorem}
\begin{proof}
  Fix some $i$ and let $Y$ be a subspace with $X_i \leq Y$ and $\dim Y = \dim X_i + 1.$
  We will show that $r(Y) > r(X_i)$.
  Recall that for any subspace $Z$ we have
  \begin{equation}\label{eq:rZ}
    r(Z) = \min_{A \leq Z} f(A) + \dim Z - \dim A,
  \end{equation}
  where function $f$ is defined as in the proof of \cref{th:transversals-are-matroids}.  Applying this to $Y$, we see that there is some subspace $B \leq Y$ with
  \[ r(Y) = f(B) + \dim Y - \dim B. \]
    If $B \leq X_i$, then by \eqref{eq:rZ}
    \begin{align*}
      r(X_i) &\leq f(B) + \dim X_i - \dim B \\
      &=f(B) + (\dim Y-1) - \dim B \\
      &= r(Y) - 1,
    \end{align*}
    proving the claim in this case.  So
    now suppose $B$ is not below $X_i$.
    There exists some subspace $H$ of codimension 1 so that $X_i = Y \cap H$.  Let $A = B \cap H$, so that
    $A \leq X_i$ and $\dim A = \dim B - 1$.  Now $r_i(A) = 0$ and $r_i(B) = 1$, so that $f(B) \geq 1 + f(A)$.
Again by  \eqref{eq:rZ} we have
    \begin{align*}
      r(X_i) &\leq f(A) + \dim X_i - \dim A \\
      &\leq (f(B) -1) + (\dim Y-1) - (\dim B -1)\\
      &= r(Y) - 1, 
    \end{align*}
    showing $r(Y) > r(X_i)$ in all cases. Therefore  $X_i$ is a flat of $M$.
\end{proof}

\begin{definition} Let ${\cal X} = (X_i \st i \in I)$ be a family of subspaces, and let $M$ be the matroid of partial $q$-transversals of $\cal X$.  Suppose for every family ${\cal Y} = (Y_i \st i \in I)$, where for all $i \in I$, $Y_i$ is a subspace of $X_i$, if $\cal Y$ is also a presentation of $M$ then for all $i$ we have $Y_i = X_i$.  Then $\cal X$ is a \emph{minimal presentation} of $M$.
\end{definition}
These minimal presentations correspond the maximal presentations of ordinary transversals.  Clearly any $q$-transversal matroid has at least one minimal presentation.

\begin{theorem}
  Let  ${\cal X} = (X_i \st i \in I)$ be a presentation of matroid $M$.  Then $\cal X$ is minimal for $M$ iff each $X_i$ is cyclic in $M$.
\end{theorem}
We know by \cref{th:q-transversal-flats} that each $X_i$ is flat in $M$, so we could equally well assert that each $X_i$ is a cyclic flat in the statement of this theorem.
\begin{proof}
  Suppose first that each $X_i$ is cyclic.
  Consider a family ${\cal Y} = (Y_i \st i \in I)$ with $Y_i \leq X_i$ for $i \in I$, and suppose for some $j \in I$ we have $Y_j < X_j$.  Let $M'$ be the $q$-matroid of partial $q$-transversals of $\cal Y$.  As in previous proofs, let $r_i$ be the rank function for the rank-1 matroid with loop space $X_i$, and similarly let $r'_i$ be the rank function for the rank-1 matroid with loop space $Y_i$.  Then $M$ is induced by function $f(S) = \sum_{i \in I} r_i(S)$ and
  $M'$ is induced by function $f'(S) = \sum_{i \in I} r'_i(S)$.  We have $r_i(S) \leq r'_i(S)$ for $i \in I$ and $S \leq V$.

  As $X_j$ is cyclic and $Y_j$ is a proper subspace, there must be a circuit $C$ of $M$ with $C \leq X_j$ with $C \not\leq Y_j$.  For all $T < C$ we have $f'(T) \geq f(T) \geq \dim T$.  We also have $r_j(C) = 0$ and $r'_j(C) = 1$, so that $f'(C) \geq 1 + f(C) = \dim C$.  Thus $C$ is independent in $M'$, and $M' \not= M$.  As this holds for all choices of  $\cal Y$, we see that $\cal X$ is a minimal presentation for $M$, giving the ``if'' part of the theorem.

  Suppose now that some $X_j$ is not cyclic, and let $Y_j < X_j$ be the join of all circuits under $X_j$.  For $i \in I \setminus \{j\}$, let $Y_i = X_i$.  Then with rank functions defined as above, for all $S \leq V$ we have $f'(S) \geq f(S)$, so all sets independent in $M$ are also independent in $M'$.  If $C$ is any circuit of $M$, we also have $f'(C) = f(C)$, so that every dependent set of $M$ is dependent in $M'$.  So $M' = M$, and $\cal X$ is not a minimal presentation.
\end{proof}

\begin{conjecture}
  Minimal presentations are unique.
\end{conjecture}

\subsection{Representation}\label{sec:rep}

We can show that in a special case, where the subspaces $X_i$  align in a particular way, a $q$-transversal matroid is representable.

\begin{theorem} \label{th:q-rep-aligned}
  Let $V$ have basis $\{b_1, b_2, \dotsc, b_n\}$ and let ${\cal X} = (X_1, \dotsc, X_k)$ be a system of subspaces of $V$ where each of the $X_i$ has the form $X_i = \left< b_i \st i \in L_i \right>$ for some sets $L_i \subseteq \{1 \upto n \}$. Then the associated $q$-transversal matroid is representable.
\end{theorem}
\begin{proof}
Let $y_1, y_2, \dotsc, y_k$ be distinct variables; we will work in the fraction field $K = \FF_q(y_1, \dotsc, y_k)$.
Let $G$ be the $k \times n$ matrix with
\[ G_{ij} = \begin{cases} 0 & \text{if $j \in L_i$,}\\ y_i^j & \text{otherwise.} \end{cases} \]
and let $g_i$ be the $i$th row of $G$.
We have the property $g_i\cdot v = 0$ iff $v \in X_i$ for any vector $v \in V$ as the powers of $y_i$ are linearly independent over $\FF_q$.

Now let $U \leq V$ be a subspace of dimension $k$, and let $\{u_1, \dotsc, u_k\}$ be a basis for $U$.  Let $S$ be a $k \times n$ matrix where the $i$th row is are the coordinates of $u_i$ in basis $(b_1, \dotsc, b_n)$.
We will show that the rank of $H = GS^T$ is $k$ iff $U$ is a $q$-transversal.
Note that $H$ is a $k \times k$ matrix and so has rank $k$ iff it has nonzero determinant.  Note too that the choice of basis for $U$ is irrelevant, so we are free to change it if necessary for the proof.

  Suppose first that $U$ is not a $q$-transversal, and suppose that $\{u_1, \dotsc, u_k\}$ is one of the bases that is not an avoidance transversal of the $X_i$.
Consider some summand
  \[ (-1)^{\sgn(\sigma)} \prod_{i=1}^n h_{i,\sigma(i)} \]
  in the formula for the determinant of $H$, where $\sigma$ is a permutation.   As the $u_i$ are not an avoidance transversal, for some $i$ we have $u_{\sigma(i)} \in X_i$.
  So $h_{i,\sigma(i)} = g_i \cdot u_{\sigma(i)} = 0$, and the product is 0.  All of the summands are thus 0, and the determinant is 0; $H$ has rank less than $k$.

  Suppose now that the rank of $H$ is less than $k$; we will show that $U$ has a vector basis that is not an avoiding transversal.  The determinant of $H$ is assumed to be zero.  Consider the coefficient of the multinomial $\prod_{i=1}^k y_i^{e_i}$ for exponents $e_i \in \{1 \upto n\}$ in the expansion of $\det H$; as these multinomials are linearly independent over $\FF_q$ this coefficient must be zero.
When some $e_i \in L_i$ then as there is no occurrence of $y_i^{e_i}$ in $G$, the coefficient of $\prod_{i=1}^k y_i^{e_i}$ is 0.
  When $e_i \notin L_i$ for all $i$ we can compute the coefficient as $\det M$, where
    \[ M = \left(\begin{array}{cccc} S_{1,e_1} & S_{2,e_1} & \dots & S_{k,e_1} \\
      S_{1,e_2} & S_{2,e_2} & \dots & S_{k,e_2} \\
      \vdots & \vdots & \cdots & \vdots \\
      S_{1,e_k} & S_{2,e_k} & \dots & S_{k,e_k} \\
    \end{array} \right) \]
    as $y_i^{e_i}$ only appears from terms in the $i$th row of $G$.  The rows of this $M$ are formed from columns $\{e_1, e_2, \dotsc, e_k\}$ of matrix $S$.  

    Let $N$ be the matroid on $\{1 \upto n\}$ represented by matrix $S$ (so that a set $T \subseteq \{1 \upto n\}$ is independent iff the columns of $S$ indexed by $T$ are linearly independent), and let ${\cal A} = (A_1, A_2, \dotsc A_k)$, where $A_i = \{1 \upto n\} \setminus L_i$.
    Then by Rado's theorem, ${\cal A}$ has a transversal that is independent in $N$ iff for all $J \subseteq \{1 \upto k\}$ we have
    \begin{equation}\label{eq:rado-L}
      \rho(A[J]) \geq |J|.
    \end{equation}
    However, ${\cal A}$ cannot have a transversal that is independent in $N$: if $x_1, \dotsc, x_k$ were such a transversal with $x_i \in A_i$ for all $i$,  then the coefficient of  $\prod_{i=1}^k y_i^{x_i}$ equals the determinant of a non-singular matrix and would not be 0.  Thus there is some  $J \subseteq \{1 \upto k\}$ where \eqref{eq:rado-L} fails, that is,  $\rho(A[J]) < |J|$.
    Consider now the submatrix $R$ of $S$ formed by taking the columns with indices in $A[J]$.  As this has rank less than $|J|$, there is some invertible matrix $C$ so that $CR$ has fewer than $|J|$ nonzero rows, and thus at least
    $k - |J|$ zero rows. So $S' = CS$, which is another vector basis for $U$, has at least  $k - |J|$ rows that are
    zero for all the columns corresponding to elements of $A[J]$.  These rows are vectors that lie in ${\cal X}(J)$,
    so we have
    \[\dim  (U \cap {\cal X}(J) > k - |J|, \]
    and thus \cref{th:is-q-transversal} shows that $U$ is not a $q$-transversal of $\cal X$.
\end{proof}
This construction gives a representation over a transcendental extension of $\FF_q$.
Inspection of the proof shows that we could use a finite extension of degree $n^k$: if $\alpha$ is an element of degree $n^k$ over $\FF_q$ we can use $\alpha^{jn^i}$ instead of $y_{i}^j$ in the representation constructed in the proof.
In experiments it is possible to find representations over extensions of even smaller degree.  More refinement of this result is surely possible.

\begin{conjecture}
  All $q$-transversal matroids are representable.
\end{conjecture}

\section{Conclusions}

We have shown a definition of a $q$-transversal that gives reasonable analogs of the main theorems of conventional transversal theory.  In a few cases analogs are only conjectured, leaving room for future work to find proofs of these properties.  Furthermore, 
we know that the class of transversal matroids is not closed under taking minors, and the smallest minor-closed class containing the transversal matroids are the gammoids.  What is the $q$-analog of that class?

\bibliographystyle{plain}
\bibliography{transversals}

\end{document}